\documentclass[11 pt]{amsart}

\usepackage{lineno,hyperref}
\usepackage{amsmath,amssymb}
\usepackage{lineno}
\usepackage{centernot}
\theoremstyle{plain}
\newtheorem{thm}{Theorem}[section]
\newtheorem{lemma}[thm]{Lemma}
\newtheorem{cor}[thm]{Corollary}
\newtheorem{prop}[thm]{Proposition}

\theoremstyle{definition}
\newtheorem{defn}[thm]{Definition}
\newtheorem{example}[thm]{Example}

\numberwithin{equation}{section}

\newcommand{\interior}[1]{{\kern0pt#1}^{\mathrm{o}}}

\begin{document}


\title{LNZS rings}

\author[Sanjiv Subba]{Sanjiv Subba $^\dagger$}

\address{$^\dagger$Department of Mathematics\\ NIT Meghalaya\\ Shillong 793003\\ India}
\email{sanjivsubba59@gmail.com}
\author[Tikaram Subedi]{Tikaram Subedi  {$^{\dagger *}$}}
\address{$^\dagger$Department of Mathematics\\ NIT Meghalaya\\ Shillong 793003\\ India}
\email{tikaram.subedi@nitm.ac.in}

\subjclass[2010]{Primary 16U80; Secondary 16S34, 16S36}.


\begin{abstract}
	This paper introduces a class of rings called \textit{left nil zero semicommutative} rings ( \textit{LNZS} rings ), wherein a ring $R$ is said to be LNZS if the left annihilator of every nilpotent element of $R$ is an ideal of $R$. It is observed that reduced rings are LNZS but not the other way around. So, this paper provides some conditions for an LNZS ring  to be reduced, and   among other results, it is proved that $R$ is reduced if and only if $T_2(R)$ is LNZS. Furthermore, it is shown through examples
	that neither the polynomial ring  nor the skew polynomial ring over an LNZS is LNZS. Therefore, this paper investigates the LNZS property of the polynomial extension and skew polynomial extension of an LNZS ring with some additional conditions.

\noindent \textbf{Keywords:} LNZS rings, Nilpotent elements, Semicommutative rings, Reduced rings.
\end{abstract}

\maketitle

\section{Introduction}
\label{intro}
 Semicommutative rings and their generalizations play an important role in non commutative ring theory. Different authors over the last several years have studied the extensions of semicommutative rings using different tools and strategies (see \cite{n2}, \cite{weaklysemi}, \cite{csemi} and \cite{wsem}). In a semicommutative ring $R$, the left annihilator of every element of $R$ is an ideal.  Therefore, it is of interest to see how a ring behaves if this property is satisfied by some elements of the ring and not necessarily every element of the ring. In this context,   this paper studies rings in which the left annihilator of every nilpotent element is an ideal.
  
   \quad All rings considered in this paper are associative with identity unless otherwise mentioned. $R$ represents a ring, and all modules are unital. The symbols $Z(R),~E(R),~ J(R),~N(R)$, and  $T_n(R)$ respectively denote the set of all central elements of $R$, the set of all idempotent elements of $R$, the Jacobson radical of $R$, the set of all nilpotent elements of $R$, and the ring of upper triangular matrices of order $n\times n$ over $R$. The notations $l(x)$ stands for the left annihilator of an element $x$ of $R$.
   
    \quad Let $ME_l(R)=\{e\in E(R)~|~Re$~ is~ a~minimal~left~ideal~of~$R\}$. An element $e\in E(R)$ is said to be \textit{left (right) semicentral} if $re=ere~(er=ere)$ for all $r\in R$. Following \cite{qn}, $R$ is called \textit{left min-abel} if every element of $ME_l(R)$ is left semicentral in $R$. $R$ is called left $MC2$ if $aRe=0$ implies $eRa=0$ for any $a\in R,~e\in ME_l(R)$. $R$ is said to be $NI$ if $N(R)$ is an ideal of $R$. Following \cite{nili}, a left $R$-module $M$ is called $Wnil-injective$ if for any $w~(\neq 0)\in N(R)$, there exists a positive integer $m$ such that $w^m\neq 0 $ and any left $R$-homomorphism $f:Rw^m\rightarrow M$ extends to one from $R$ to $M$. Now, recall that $R$ is said to be:
    \begin{enumerate}
     \item \textit{reduced} if $N(R)=0$.
     \item \textit{reversible} (\cite{ext}) if $wh=0$ implies $hw=0$ for any $w,~h\in R$.
     \item  \textit{semicommutative} (\cite{weaklysemi}) if $wh=0$ implies $wRh=0$ for any $w,~h\in R$.
     \item \textit{weakly semicommutative} (\cite{weaklysemi}) if for any $w,h\in R$ satisfy $wh=0$ then $wrh\in N(R)$ for any $r\in R$.
     \item \textit{abelian} if $E(R)\subseteq Z(R)$.
     \end{enumerate}
\section{LNZS rings}
\label{sec:1}
\begin{defn}\label{def11}
We call a ring $R$ left nil zero semicommutative ( LNZS ) if $l(a)$ is an ideal for any $a\in N(R)$.
\end{defn}
Clearly, semicommutative rings are LNZS. Nevertheless, not every LNZS ring is semicommutative; for example, take $R=T_2(\mathbb{Z}_2)$. It is easy to see that $R$ is not semicommutative, whereas $R$ is LNZS in view of the Proposition \ref{tri}.

\begin{thm}\label{tri}
$T_2(R)$ is LNZS if and only if $R$ is reduced.
\end{thm}
\begin{proof}
Suppose $T_2(R)$ be an LNZS ring. On the contrary, we assume that $R$ is not reduced. Then $x^2=0$ for some $x~(\neq 0)\in R$. So $\left[\begin{array}{rr}
x & x+1 \\
0 & -x
\end{array} \right]\in N(T_2(R))$. Observe that $\left[\begin{array}{rr}
x & 1 \\
0 & x
\end{array} \right]\left[\begin{array}{rr}
x & x+1 \\
0 & -x
\end{array} \right]=\left[\begin{array}{rr}
0 & 0 \\
0 & 0
\end{array} \right]$. Since $T_2(R)$ is LNZI, $\left[\begin{array}{rr}
x & 1 \\
0 & x
\end{array} \right]\left[\begin{array}{rr}
1 & 1 \\
0 & 0
\end{array} \right]\left[\begin{array}{rr}
x & x+1 \\
0 & -x
\end{array} \right]=\left[\begin{array}{rr}
0 & 0 \\
0 & 0
\end{array} \right]$, that is, $\left[\begin{array}{rr}
0 & x \\
0 & 0
\end{array} \right]=\left[\begin{array}{rr}
0 & 0 \\
0 & 0
\end{array} \right]$. This implies that $x=0$, a contradiction. Therefore $R$ is reduced.\\
Conversely, assume that $R$ is reduced. Then $N(T_2(R))=\left[\begin{array}{rr}
0 & R \\
0 & 0
\end{array} \right]$. Let $\left[\begin{array}{rr}
x & y \\
0 & z
\end{array} \right]\in T_2(R)$, $\left[\begin{array}{rr}
0 & a \\
0 & 0
\end{array} \right]\in N(T_2(R))$ be such that $\left[\begin{array}{rr}
x & y \\
0 & z
\end{array} \right]\left[\begin{array}{rr}
0 & a \\
0 & 0
\end{array} \right]=\left[\begin{array}{rr}
0 & 0 \\
0 & 0
\end{array} \right]$, that is, $xa=0$. Since $R$ is reduced, $xha=0$ for any $h\in R$. For any arbitrary element $\left[\begin{array}{rr}
s & t \\
0 & w
\end{array} \right]\in T_2(R)$, we have $\left[\begin{array}{rr}
x & y \\
0 & z
\end{array} \right]\left[\begin{array}{rr}
s & t \\
0 & w
\end{array} \right]\left[\begin{array}{rr}
0 & a \\
0 & 0
\end{array} \right]=\left[\begin{array}{rr}
0 & xsa \\
0 & 0
\end{array} \right]=\left[\begin{array}{rr}
0 & 0 \\
0 & 0
\end{array} \right]$. Thus, $l\left(\left[\begin{array}{rr}
0 & a \\
0 & 0
\end{array} \right]\right) $ is an ideal. So $T_2(R)$ is LNZS.
\end{proof}
Proposition \ref{tri} cannot be extended to $T_n(R)$ for any integer $n\geq 3$ ( see the following example ).
\begin{example}\label{T3}
Take the ring  $S=T_3(R)$, where $R$ denotes a non zero reduced ring. Observe that $e_{23}\in N(T_3(R))$ and $e_{11}e_{23}=0$ but $e_{11}e_{12}e_{23}\neq 0$, that is, $l(e_{23})$ is not an ideal. So $T_3(R)$ is not LNZS. Hence, for any non zero reduced ring $R$, $T_n(R)$ is not LNZS for $n\geq 3$.
\end{example}
\begin{prop}\label{ni}
LNZS rings are $NI$.
\end{prop}

\begin{proof}
Suppose $w\in N(R)$. Then $w^n=0$ for some positive integer $n$. Since $R$ is LNZS, $wrw^{n-1}=0$ for any $r\in R$. Proceeding similarly, we get $(wr)^n=0$. Therefore, $wr, rw\in N(R)$. Let $h\in R$ be such that $h^m=0$ for some positive integer $m$. Take $l=m+n+1$. Then $(w+h)^l=\sum\limits_{i_1+j_1+...+i_p+j_p=l}(w^{i_1}h^{j_1}w^{i_2}h^{j_2}...w^{i_p}h^{j_p})$, $0\leq i_1,~j_1,...,~i_p,~j_p\leq l$, $p$ is some positive integer. If $i_1+i_2+...+i_p\geq n$, then $w^{i_1}...w^{i_p}=0$. Since $R$ is LNZS, $w^{i_1}h^{j_1}w^{i_2}h^{j_2}...w^{i_p}h^{j_p}=0$. If $i_1+i_2+...+i_p\leq n$, then $j_1+j_2...+j_p\geq m$. So $h^{j_1}h^{j_2}...h^{j_p}=0$. As above, $w^{i_1}h^{j_1}w^{i_2}h^{j_2}...w^{i_p}h^{j_p}=0$. Thus, $(w+h)^l=0$. Hence $R$ is NI.
\end{proof}

\begin{cor}\label{weak}
LNZS rings are weakly semicommutative.
\end{cor}
\begin{proof}
Follows since NI rings are weakly semicommutative (\cite[Proposition 2.1]{wsem}).
\end{proof}
There exists a weakly semicommutative ring $R$ that is not LNZS ( see the following example ).
\begin{example}

 Let $F$ be a field and $F<x,y>$ the free algebra in non-commuting indeterminates $x,y$ over $F$ and  $I$ denotes the ideal $<x^2>^2$ of $F<x,y>$, where $<x^2>$ is the ideal of $F<x,y>$ generated by $x^2$. Take $R=F<x,y>/I$. Then by \cite[Example 2.1]{wsem}, $R$ is weakly semicommutative and $N(R)=\overline{x}R\overline{x}+R\overline{x}^2R+F\overline{x}$. Hence $\overline{x}\in N(R)$ and $ (\overline{x}^3)\overline{x}=0$. But $\overline{x}^3\overline{y}(\overline{x})\neq 0$. Thus, $R$ is not LNZS.
\end{example}

We write $R_n$ to  denote the ring\\
 $\left \lbrace
 \left(\begin{array}{lccccr}
 a & a_{12} & \dots & a_{1n}\\
 0 & a & \dots & a_{2n}\\
 \vdots & \vdots &\ddots & \vdots\\
 0 & 0 & \dots & a \\
 \end{array}
 \right ):a, a_{ij}\in R,~ i=1,\dots,n-1,~j=2,3,\dots,n\right \rbrace$. \\
 In \cite[Example 2.1]{weaklysemi}, it is shown that if $R$ is reduced then, $R_n$ is not semicommutative but weakly  semicommutative for $n\geq 4$. So a suspicion arises whether $R_n$ is LNZS for $n\geq 4$ whenever $R$ is reduced. However, the following example obliterates the possibility.
 \begin{example}\label{R4}
 Let $R=\mathbb{R}$, $R_4=\left\{\left(\begin{array}{rrrr}
 a & b & c & d\\
 0 & a & e & f\\
 0 & 0 & a & g\\
 0 & 0 & 0 & a
 \end{array}\right)
 :a,b,c,d,e,f,g\in R \right\} $, where $\mathbb{R}$ represents the field of real numbers. Take $A=\left( \begin{array}{rrrr}
 0 & 1 & 0 & 1\\
 0 & 0 & 0 & 0\\
 0 & 0 & 0 & 1\\
 0 & 0 & 0 & 0
 \end{array}\right)\in N(R_4)$. \\ Now, $E=\left( \begin{array}{rrrr}
 0 & 1 & 0 & 1\\
 0 & 0 & 0 & 1\\
 0 & 0 & 0 & 1\\
 0 & 0 & 0 & 0
 \end{array}\right)\in l(A)$ i.e., $EA=0$. Take $F=\left( \begin{array}{rrrr}
 1 & 1 & 1 & 1\\
 0 & 1 & 1 & 1\\
 0 & 0 & 1 & 1\\
 0 & 0 & 0 & 1
 \end{array}\right)$. Then, $EF=\left( \begin{array}{rrrr}
 0 & 1 & 0 & 1\\
 0 & 0 & 0 & 1\\
 0 & 0 & 0 & 1\\
 0 & 0 & 0 & 0
 \end{array}\right)\left( \begin{array}{rrrr}
 1 & 1 & 1 & 1\\
 0 & 1 & 1 & 1\\
 0 & 0 & 1 & 1\\
 0 & 0 & 0 & 1
 \end{array}\right)=\left( \begin{array}{rrrr}
 0 & 1 & 1 & 2\\
 0 & 0 & 0 & 1\\
 0 & 0 & 0 & 1\\
 0 & 0 & 0 & 0
 \end{array}\right)$ and $EFA=\left( \begin{array}{rrrr}
 0 & 0 & 0 & 1\\
 0 & 0 & 0 & 0\\
 0 & 0 & 0 & 0\\
 0 & 0 & 0 & 0
 \end{array}\right)$. Thus, $EF\notin l(A)$, that is, $l(A)$ is not an ideal. Therefore, $R$ is not LNZS.
 \end{example}
  \begin{thm}
   Let $R$ be a ring. Then $R$ is a domain if and only if $R$ is an LNZS and prime ring.
   \end{thm}
    
  \begin{proof}
    The necessary part is obvious. Conversely, assume $R$ is an LNZS and prime ring and $w,h\in R$ be such that $wh=0$. Since $R$ is LNZS and $(hw)^2=0$, $hwRhw=0$. By hypothesis, $hw=0$. So for any $r\in R$, $(wrh)^2=0$ which further implies that $wrhRwrh=0$, that is, $wrh=0$. Hence $w=0$ or $h=0$.
    \end{proof}
     One might suspect that a homomorphic image of an LNZS ring is LNZS. However, there exists an LNZS ring whose homomorphic image is not LNZS ( see the following example ).
              \begin{example}
              Let $R=D[x,y,z]$, where $D$ is a division ring and $x,~y$, and $z$ are non-commuting indeterminates. Take $I=<xy>$. As $R$ is a domain, $R$ is LNZS. Clearly, $\overline{yx}\in N(R/I)$ and $\overline{x}(\overline{yx})=\overline{xyx}=0$. But, $(\overline{x})(\overline{z})\overline{yx}\neq 0$. So $R/I$ is not LNZS. 
              
              \end{example}
     \begin{thm}\label{qusem}
       Let $R$ be an $LNZS$ ring. Then $R/l(s)$ is semicommutative for any $s\in N(R)$.
            \end{thm}
     \begin{proof}
     Let $a$ be an element of $R$ and $x\in l(as)$. By Proposition \ref{ni}, $N(R)$ is an ideal and for any $r\in R$, $xr\in l(as)$ as $R$ is LNZS. So $\overline{x}\overline{r}\in l(\overline{a})$, that is, $l(\overline{a})$ is an ideal.
     \end{proof} 
     
      \begin{thm}
       If $R$ is an $LNZS$ ring and $H$ an ideal consisting of all nilpotent elements of bounded index ~$ m$ in $R$, then $R/H $ is $LNZS$.
      \end{thm}
       \begin{proof}
       Let $\overline{b}\in N(R/H),\overline{a} \in R/H$ be such that $\overline{a}\overline{b}=0$. Then $ab\in H$, that is, $(ab)^m=0$. As $\overline{b}\in N(R/H)$, $b\in N(R)$. By Proposition \ref{ni} and $R$ being LNZS, $(ab)^m=0$ implies $(arb)^m=0$ for any $r\in R$. This implies $arb\in H$. Therefore, $R/H$ is LNZS.
      \end{proof}

        Let $S$ be an $(R,R)$-bimodule. The \textit{trivial extension} of $R$ by $S$ is the ring $T(R,S)=R\bigoplus S$, where the addition is usual and the multiplication is defined as:\\
                 $(r_1,s_1)(r_2,s_2)=(r_1r_2,r_1s_2+s_1r_2),s_i\in S,r_i\in R$ and $i=1,2$. $T(R,S)$ is isomorphic to the ring  $\left\{\left(\begin{array}{rr}
                 t & s\\
                 0 & t
                 \end{array}
                 \right) : t\in R, s\in S \right\}$, where the operations are usual matrix operations.
       \begin{thm}
        If $T(R,R)$ is $LNZS$, then $R$ is semicommutative.
       \end{thm} 
      \begin{proof}
      Let $w,~h\in R$ and $wh=0$. Then, $\begin{pmatrix}
      w & 0\\
      0 & w
      \end{pmatrix}\begin{pmatrix}
       0 & h\\
       0 & 0
       \end{pmatrix}=\begin{pmatrix}
        0 & 0\\
        0 & 0
        \end{pmatrix}$. Since $T(R,R)$ is LNZS, $\begin{pmatrix}
         w & 0\\
         0 & w
         \end{pmatrix}\begin{pmatrix}
          r & 0\\
          0 & r
          \end{pmatrix}\begin{pmatrix}
           0 & h\\
           0 & 0
           \end{pmatrix}=\begin{pmatrix}
            0 & 0\\
            0 & 0
            \end{pmatrix}$, for any $r\in R$. This implies that \\ $\begin{pmatrix}
             0 & wrh\\
             0 & 0
             \end{pmatrix}=\begin{pmatrix}
              0 & 0\\
              0 & 0
              \end{pmatrix}$, that is, $wrh=0$.
      \end{proof}  
  However, the converse is not true.
      \begin{example}
      Consider $\mathbb{H}$ as the Hamilton quaternions over $\mathbb{R}$ and $R=T(\mathbb{H},\mathbb{H})$, where $\mathbb{R}$ denotes the field of real numbers. By \cite[Proposition 1.6]{ext}, $R$ is reversible and hence semicommutative. Take $S=T(R,R)$. Let $A=\begin{pmatrix}
      \begin{pmatrix}
      0 & i\\
      0 & 0
      \end{pmatrix} & \begin{pmatrix}
       j & 0\\
       0 & j
       \end{pmatrix}\\
       \begin{pmatrix}
        0 & 0\\
        0 & 0
        \end{pmatrix} & \begin{pmatrix}
         0 & i\\
         0 & 0
         \end{pmatrix}
      \end{pmatrix}\in S$ and $B=\begin{pmatrix}
       \begin{pmatrix}
       0 & 1\\
       0 & 0
       \end{pmatrix} & \begin{pmatrix}
        k & 0\\
        0 & k
        \end{pmatrix}\\
        \begin{pmatrix}
         0 & 0\\
         0 & 0
         \end{pmatrix} & \begin{pmatrix}
          0 & 1\\
          0 & 0
          \end{pmatrix}
       \end{pmatrix}\in N(S)$. Observe that $A\in l(B)$. Take \\ $C=\begin{pmatrix}
        \begin{pmatrix}
        j & i\\
        0 & j
        \end{pmatrix} & \begin{pmatrix}
         0 & 0\\
         0 & 0
         \end{pmatrix}\\
         \begin{pmatrix}
          0 & 0\\
          0 & 0
          \end{pmatrix} & \begin{pmatrix}
           j & i\\
           0 & j
           \end{pmatrix}
        \end{pmatrix}$. It is easy to see that $ACB\neq 0$. Hence $l(B)$ is not an ideal, that is, $S$ is not LNZS.
      \end{example}
     According to \cite{qn}, $R$ is called \textit{quasi-normal} if $eR(1-e)Re=0$ for each $e\in E(R)$.
     \begin{thm}\label{qnor}
     $LNZS$ rings are quasi-normal.
     \end{thm}
     \begin{proof}
     Let $r$ be an arbitrary element in $R$. Then $e(1-e)re=0$. Clearly, $(1-e)re\in N(R)$. Since $R$ is LNZS, $es(1-e)re=0$ for any $s\in R$. Thus, $eR(1-e)Re=0$.
     \end{proof}  
     However, there exists a quasi-normal ring which is not LNZS, as shown in the following example.
     \begin{example}
        Let $\mathbb{Z}$ be the ring of integers, and consider the ring $R$=\\
        $\left\{\left(\begin{array}{rr}
        x & y \\
        z & w\\ 
        \end{array}\right): x\equiv w~(mod~ 2), y\equiv z~(mod ~2),x,y,z,w\in \mathbb{Z} \right \}$. By \cite[Example 2.7]{csemi}, $R$ is abelian and hence quasi-normal. Observe that, $\left(\begin{array}{rr}
          0 & 2 \\
          0 & 0\\
          \end{array} \right)^2=\left (\begin{array}{rr}
           0 & 0\\
           0 & 0\\
           \end{array} \right)$. But \\ $\left (\begin{array}{rr}
             0 & 2 \\
             0 & 0\\
             \end{array} \right)\left (\begin{array}{rr}
              2 & 2 \\
              2 & 2\\
              \end{array} \right)\left (\begin{array}{rr}
                   0 & 2 \\
                   0 & 0\\
                   \end{array} \right)\neq \left (\begin{array}{rr}
               0 & 0\\
               0 & 0\\
               \end{array} \right)$. So, $R$ is not LNZS.
     \end{example}       
\begin{cor}\label{lma}
LNZS rings are left min-abel.
\end{cor}     

\begin{proof}
 Quasi normal rings are left min-abel (\cite[Theorem 2.4]{qn}).
\end{proof}
 It is easy to observe that the class of LNZS rings properly contains the class of reduced rings. In the following proposition, we provide some conditions under which an LNZS ring turns out to be reduced.
  \begin{prop}\label{gpvr}
  Let $R$ be an LNZS ring. Then $R$ is reduced in each of the following cases:
  \begin{enumerate}
  \item $R$ is semiprime.
  \item R is left MC2, and every simple singular left R-module is Wnil-injective. 
  \item Every idempotent in $R$ is right semicentral, and every simple singular left $R$-module is $Wnil$-injective.
  \end{enumerate}
  \end{prop}
  \begin{proof}
  Let $R$ be an LNZS ring.
  \begin{enumerate}
  \item  Assume that $w^2=0$ for some $w\in R$. Since $R$ is LNZS, $wRw=0$. By hypothesis, $w=0$. Therefore, $R$ is reduced.
  
  \item Suppose that there exists $h~(\neq 0)\in R$ with $h^2=0$. Then $l(h)\subseteq H$ for some maximal left ideal $H$. If possible, assume that $H$ is not an essential left ideal of $R$. Then $H=l(e)$ for some $e\in ME_l(R)$. By Corollary \ref{lma}, $R$ is left min-abel. By \cite[Theorem 1.8 (3)]{mc2}, $e\in Z(R)$. As $h\in H$, $eh=0$ and so $e\in l(h)\subseteq H=l(e)$, a contradiction. Hence $H$ is an essential left ideal of $R$, and $R/H$ is a simple singular left $R$-module. By hypothesis, $R/H$ is $W$nil-injective. Define a left $R$-homomorphism $\Psi:Rh\rightarrow R/H$ via $\Psi  (rh)=r+H$. Then $\Psi$ can be extended from $R$ to R/H. So, $1-hl\in H$ for some $l\in R$. Since $R$ is LNZS, $h^2=0$ yields $hRh=0$. So $hl\in l(h)\subseteq H$, that is, $1\in H$, a contradiction. Therefore $R$ is reduced.
  \item Let $w~(\neq 0)\in R$ with $w^2=0$. Then $l(w)\subseteq W$ for some maximal left ideal $W$ of $R$. If possible, assume that $W$ is not essential in $_RR$. Then $W=l(e)$ for some $e\in E(R)$. So, $we=0$ as $w\in l(w)\subseteq W=l(e)$. Since $e$ is right semicentral, $ew=ewe=0$. Hence $e\in l(w)\subseteq l(e)$, a contradiction. So $W$ is essential maximal left ideal. This implies $R/W$ is simple singular left $R$-module. By hypothesis $R/W$ is $W$nil-injective left $R$-module. As in $(2)$, $1-ws\in W$ for some $s\in R$. Since $R$ is LNZS, $w^2=0$ implies $wRw=0$. So, $ws\in l(w)\subseteq W$. This implies $1\in W$, a contradiction. Thus, $w=0$.
  
  \end{enumerate}
  \end{proof} 
    \begin{prop}\label{subdi}
           The subdirect product of an arbitrary family of LNZS rings is LNZS.
       \end{prop}
       \begin{proof}
        Let $\{I_{\delta}| \delta\in \Delta \}$ be a family of ideals of $R$ such that $\cap_{\delta \in \Delta}I_\delta=0$ and $R/I_\delta$ is LNZS for each $\delta$, where $\Delta$ is an index set. Take $b\in N(R)$ and $a\in l(b)$. Clearly, for each $\delta$,  $b+I_\delta \in N(R/I_\delta)$ and $a+I_\delta\in l(b+I_\delta)$. Since $R/I_\delta$ is LNZS and $(a+I_\delta)(x+I_\delta)(b+I_\delta)=I_{\delta}$ for any $x\in R$, that is, $axb+I_\delta=I_{\delta}$. This implies that $axb\in I_\delta$ for all $\delta\in \Delta$, that is, $axb\in \cap_{\delta \in \Delta}I_\delta=0$ and so $l(a)$ is an ideal. Therefore, the subdirect product of LNZS rings is LNZS.
       \end{proof}
     Let $A$ be an algebra (not necessarily with identity) over a commutative ring $S$. The \emph{Dorroh extension} of $A$ by $S$ is the ring denoted by $A \bigoplus _{D} S$, with the operations $(a,s)+(a_1,s_1)=(a+a_1,s+s_1)$ and $(a,s)(a_1,s_1)=(a a_1+sa_1+s_1a,ss_1)$ where $a, a_1\in A$ and $s, s_1\in S$.
     
      \begin{prop}
          Dorroh extension of an $LNZS$ ring $R$ by $\mathbb{Z}$ is LNZS.
          \end{prop}
         \begin{proof}
         Let $(a,z)\in N(R\bigoplus_D\mathbb{Z})$. Clearly, $a\in N(R)$ and $z=0$. Suppose $(s,m)\in l(a,0)$. Then $((s+m)a,0)=(0,0)$, that is, $(s+m)\in l(a)$. Let $(r,n)$ be an arbitrary element of $R\bigoplus_D\mathbb{Z}$. Now, $(s,m)(r,n)(a,0)=(s,m)(ra+na,0)=((s+m)ra+(s+m)na,0)$. Since $l(a)$ is an ideal of $R$ and  $(s+m)\in l(a)$, so  $(s,m)(r,n)(a,0)=(0,0)$. Hence $l(a,0)$ is an ideal.  
         \end{proof}
         
      For an endomorphism $\alpha$ of $R$, $R[x;\alpha]$ denotes the  \textit{skew polynomial} ring over $R$ whose elements are polynomials $\sum\limits_{i=0}^{n}w_ix^i,w_i\in R$, where addition is defined as  usual and the multiplication is defined by the law $xw=\alpha (w)x$ for any $w\in R$.\\
    The following examples show that the polynomial ring over an LNZS ring need not be LNZS, and so is the case of the skew polynomial ring over an LNZS ring.
       \begin{example}: 
       \begin{enumerate}
      
       \item Take $\mathbb{Z}_2$ as the field of integers modulo 2 and let $A=\mathbb{Z}_2[a_0,a_1,a_2,b_0,b_1,b_2,c]$ be the free algebra of polynomials with zero constant terms in non-commuting indeterminates $a_0,a_1,a_2,b_0,b_1,b_2$ and $c$ over $\mathbb{Z}_2.$ Take an ideal $I$ of the ring $\mathbb{Z}_2+A$ generated by $a_0b_0, a_0b_1+a_1b_0, a_0b_2+a_1b_1+a_2b_0, a_1b_2+a_2b_1, a_2b_2,a_0rb_0, a_2rb_2, b_0a_0, b_0a_1+b_1a_0, b_0a_2+b_1a_1+b_2a_0, b_1a_2+b_2a_1,b_2a_2,b_0ra_0,b_2ra_2, (a_0+a_1+a_2)r(b_0+b_1+b_2),(b_0+b_1+b_2)r(a_0+a_1+a_2) ~and~ r_1r_2r_3r_4$ where $r,r_1,r_2,r_3,r_4\in A$. Take $R=(\mathbb{Z}_2+A)/I$. Then we have $R[x]\cong (\mathbb{Z}_2+A)[x]/I[x]$. By \cite{ext}, $R$ is reversible and hence LNZS.  
       Observe that $\overline{(a_0+a_1x+a_2x^2)}\in N(R)$ and $(b_0+b_1x+b_2x^2)(a_0+a_1x+a_2x^2)\in I[x]$. But $(b_0+b_1x+b_2x^2)c(a_0+a_1x+a_2x^2)\notin I[x]$, since $b_0ca_1+b_1ca_0\notin I$. Hence $l(\overline{(a_0+a_1x+a_2x^2)})$ is not an ideal. Therefore, $R$ is not LNZS.
   \item  Let $K$ be a division ring and $R=K\bigoplus K$ with componentwise multiplication. Observe that $R$ is reduced, and so $R$ is LNZS. Define $\Psi: R\rightarrow R$ via  $\Psi(h,w)=(w,h)$. Then $\Psi$ is an automorphism of $R$. Let $f(x)=(1,0)x\in R[x;\Psi]$. Observe that $f(x)^2=0$ but $f(x)xf(x)\neq 0$, that is, $l(f(x))$ is not an ideal. Hence $R[x;\Psi]$ is not LNZS.    
        
       \end{enumerate} 
      \end{example} 
       The proof of the following lemma is trivial.
           	\begin{lemma}\label{loc}
           	Let $R$ be a ring and $\Delta $ be a multiplicatively closed subset of $R$ consisting of central non-zero divisors. For any $u^{-1}a \in  \Delta ^{-1} R$, $l(u^{-1}a)$ is an ideal if and only if $l(a)$ is an ideal.
           	\end{lemma}
           	
           	\begin{prop}
           	Let $R$ be a ring and $\Delta$ be a multiplicatively closed subset of $R$ consisting of non-zero divisors. Then $R$ is $LNZS$ if and only if $\Delta ^{-1} R$ is $LNZS$.
           	\end{prop}
           	\begin{proof}
           	Obviously, $u^{-1}a\in N(\Delta^{-1} R)$ if and only if $a\in N(R)$. By Lemma \ref{loc}, $R$ is LNZS if and only if so is $\Delta^{-1}R$. 
           	\end{proof}
       	\begin{cor}
            	$R[x]$ is $LNZS$ if and only if $R[x,x^{-1}]$ is so.
            	\end{cor}
            	
            	\begin{proof}
            	 $R[x,x^{-1}]=\Delta ^{-1}R[x]$, where $\Delta =\{1,x,x^2,...\}$. Hence the result follows.
            	\end{proof}
            	 
             $R$ is said to satisfy the $\alpha$-condition (\cite{weaklysemi}) for an endomorphism $\alpha$ of $R$ in case $ab=0$ if and only if $a\alpha(b)=0$ where $a,b\in R$.
             
             \quad For ease of reference, we state the following lemma.
             
              \begin{lemma}(\cite[Lemma 3.1]{wsem})\label{ac}
             Let $R$ be a ring which statisfies $\alpha$-condition for an endomorphism $\alpha$ of $R$. Then $a_1a_2...a_n=0\iff \alpha^{k_1}(a_1)\alpha^{k_2}(a_2)...\alpha^{k_n}(a_n)=0$, where $k_1,~k_2,...,~k_n$ are arbitrary non negative integers and $a_1,~a_2,...,a_n$ are arbitrary elements in $R$.
             \end{lemma} 
     \begin{prop}\label{nl}
        Let $R$ be an $LNZS$ ring and satisfies $\alpha$-condition for an endomorphism $\alpha$ of $R$. Then  $N(R)[x;\alpha]\subseteq N(R[x;\alpha])$.
        \end{prop}
        \begin{proof}
        Let $f(x)=a_0+a_1x+a_2x^2+...+a_nx^n\in N(R)[x;\alpha]$. Then for each $i$, there exists a positive integer $m_i$ satisfying $a_i^{m_i}=0$. Take $k=m_0+m_1+m_2+...+m_n+1$. Then, $(f(x))^k=\sum\limits_{l=0}^{kn}\left(\sum\limits_{i_1+...+i_k=l}c_l\right)x^l$, where   
        $c_l=a_{i_1}\alpha^{i_1}(a_{i_2})\alpha^{i_1+i_2}(a_{i_3})...\alpha^{i_1+i_2+...+i_{k-1}}(a_{i_k})$ and  $a_{i_1},a_{i_2},a_{i_3},...,a_{i_k}\in \{a_0,a_1,a_2,...,a_n \}$.
          There exists $a_t \in \{a_0,a_1,a_2,...,a_n \}$ such that $a_t$ appears more than $m_t$ in the expression of $c_l$ as given above. Thus, we may assume that $a_t$ appears $s>m_t$ times in $c_l$. So we can rewrite $c_l$ as; \\
        
        $c_l=b_0\alpha^{j_1}(a_t)b_1\alpha^{j_1+j_2}(a_t)...b_{s-1}\alpha^{j_1+...+j_s}(a_t)b_s$,
       
        where $b_i\in R$ for each $i$ and $j_1,~j_2,...,~j_s$ are non negative integers. Clearly $a_t^s=0$. Since $R$ is LNZS, $b_0a_tb_1a_t...b_{s-1}a_tb_s=0$. By Lemma \ref{ac},\\  $c_l=b_0\alpha^{j_1}(a_{t})b_1\alpha^{j_1+j_2}(a_{t})...b_{s-1}\alpha^{j_1+j_2+...+j_{s}}(a_{t})b_s=0$. This implies $(f(x))^k=0$ and hence $N(R)[x;\alpha]\subseteq N(R[x;\alpha])$.
        \end{proof}
        The following lemma is established in the proof of \cite[Lemma 3.4]{n2}.
        \begin{lemma}\label{nr}
          Let $R$ be an $NI$ ring and satisfies $\alpha$-condition for an endomorphism $\alpha$ of $R$. Then  $N(R[x;\alpha])\subseteq N(R)[x;\alpha]$.
          \end{lemma} 
      Proposition \ref{ni}, Proposition \ref{nl} and Lemma \ref{nr} together yield the following theorem.
       \begin{thm}\label{nilr}
          Let $R$ be an $LNZS$ ring and satisfies $\alpha$-condition for an endomorphism $\alpha$ of $R$. Then  $N(R)[x;\alpha]=N(R[x;\alpha])$.
       
       \end{thm}
    
    According to \cite{n2}, a ring $R$ is called $\alpha$-skew Armendariz for an endomorphism $\alpha$ of $R$ if for any $f(x)=\sum\limits_{i=0}^{n}w_ix^i, g(x)=\sum\limits_{j=0}^{m}h_jx^j\in R[x;\alpha]$ whenever $f(x)g(x)=0$ then $a_i\alpha^i(b_j)=0$ for all $i$ and $j$.
         
    \begin{thm}
    Let $R$ be a ring satisfying $\alpha$-condition for an endomorphism $\alpha$ of $R$. If $R$ is $\alpha$-skew Armendariz, then $R$ is $LNZS$ if and only if $R[x;\alpha]$ is $LNZS$.
    \end{thm}
    \begin{proof}
    We prove the  necessary part only while the other part follows from the closedness of LNZS rings under subrings. Let  $g(x)=b_0+b_1x+b_2x^2+...+b_mx^m\in N(R[x;\alpha])$ and $f(x)=a_0+a_1x+a_2x^2+...+a_nx^n\in l(g(x))$. Then $f(x)g(x)=0$. Since $R$ is $\alpha$-skew Armendariz, $a_i\alpha^i(b_j)=0$ for all $i,j$. By Lemma \ref{ac}, $a_ib_j=0$ for all $i,j$. Let $h(x)=c_0+c_1x+c_2x^2+...+c_px^p\in R[x;\alpha]$. By Theorem \ref{nilr}, $b_j\in N(R)$ for all $j$, so $l(b_j)$ is an ideal. Therefore, $a_ic_lb_j=0$ for all $l=0,1,...,p$. It follows from Lemma \ref{ac} that $a_i\alpha^i(c_l)\alpha^{i+l}(b_j)=0$. Hence $f(x)h(x)g(x)=0$, that is, $l(g(x))$ is an ideal. 
    \end{proof} 
     According to Rege and Chhawchharia (\cite{mb}), a ring $R$ is called \textit{Armendariz}  if  $f(x)=\sum\limits_{i=0}^{n}w_ix^i, g(x)=\sum\limits_{j=0}^{m}h_jx^j \in R[x]$ satisfy $f(x)g(x)=0$, then $w_ih_j=0$ for each $i,j$.
                            
    \begin{cor}
         For an Armendariz ring $R$, the following are equivalent:
            \begin{enumerate}
            \item $R$ is $LNZS$.
            \item $R[x]$ is $LNZS$.
            \item $R[x,x^{-1}]$ is $LNZS$.
            \end{enumerate}
         \end{cor}

\end{document}